\documentclass[10pt]{amsart}
\usepackage{amsfonts}
\usepackage{mathrsfs}
\usepackage{amscd}
\usepackage{amsmath}
\usepackage{amssymb}
\usepackage{latexsym}
\usepackage{lscape}
\usepackage{xypic}
\usepackage{comment}
\usepackage{amscd}
\usepackage{wasysym}  
\usepackage{tikz} 
\usetikzlibrary{matrix,arrows}
\usepackage{appendix}

\newtheorem{thm}{Theorem}[section]

\newtheorem{prop}[thm]{Proposition}
\newtheorem{lem}[thm]{Lemma}
\newtheorem{sublem}[thm]{Sublemma}
\newtheorem{lem-def}[thm]{Lemma-Definition}
\newtheorem{cor}[thm]{Corollary}

\theoremstyle{definition}
\newtheorem*{ack}{Acknowledgement}

\newtheorem{rmk}[thm]{Remark}
\newtheorem{dfn}[thm]{Definition}

\numberwithin{equation}{section}

\newcommand{\nc}{\newcommand}

\nc{\on}{\operatorname}
\nc{\fraka}{{\mathfrak a}}
\nc{\frakb}{{\mathfrak b}}
\nc{\frakc}{{\mathfrak c}}
\nc{\frakd}{{\mathfrak d}}
\nc{\frake}{{\mathfrak e}}
\nc{\frakf}{{\mathfrak f}}
\nc{\frakg}{{\mathfrak g}}
\nc{\frakh}{{\mathfrak h}}
\nc{\fraki}{{\mathfrak i}}
\nc{\frakj}{{\mathfrak j}}
\nc{\frakk}{{\mathfrak k}}
\nc{\frakl}{{\mathfrak l}}
\nc{\frakm}{{\mathfrak m}}
\nc{\frakn}{{\mathfrak n}}
\nc{\frako}{{\mathfrak o}}
\nc{\frakp}{{\mathfrak p}}
\nc{\frakq}{{\mathfrak q}}
\nc{\frakr}{{\mathfrak r}}
\nc{\fraks}{{\mathfrak s}}
\nc{\frakt}{{\mathfrak t}}
\nc{\fraku}{{\mathfrak u}}
\nc{\frakv}{{\mathfrak v}}
\nc{\frakw}{{\mathfrak w}}
\nc{\frakx}{{\mathfrak x}}
\nc{\fraky}{{\mathfrak y}}
\nc{\frakz}{{\mathfrak z}}
\nc{\frakA}{{\mathfrak A}}
\nc{\frakB}{{\mathfrak B}}
\nc{\frakC}{{\mathfrak C}}
\nc{\frakD}{{\mathfrak D}}
\nc{\frakE}{{\mathfrak E}}
\nc{\frakF}{{\mathfrak F}}
\nc{\frakG}{{\mathfrak G}}
\nc{\frakH}{{\mathfrak H}}
\nc{\frakI}{{\mathfrak I}}
\nc{\frakJ}{{\mathfrak J}}
\nc{\frakK}{{\mathfrak K}}
\nc{\frakL}{{\mathfrak L}}
\nc{\frakM}{{\mathfrak M}}
\nc{\frakN}{{\mathfrak N}}
\nc{\frakO}{{\mathfrak O}}
\nc{\frakP}{{\mathfrak P}}
\nc{\frakQ}{{\mathfrak Q}}
\nc{\frakR}{{\mathfrak R}}
\nc{\frakS}{{\mathfrak S}}
\nc{\frakT}{{\mathfrak T}}
\nc{\frakU}{{\mathfrak U}}
\nc{\frakV}{{\mathfrak V}}
\nc{\frakW}{{\mathfrak W}}
\nc{\frakX}{{\mathfrak X}}
\nc{\frakY}{{\mathfrak Y}}
\nc{\frakZ}{{\mathfrak Z}}
\nc{\bbA}{{\mathbb A}}
\nc{\bbB}{{\mathbb B}}
\nc{\bbC}{{\mathbb C}}
\nc{\bbD}{{\mathbb D}}
\nc{\bbE}{{\mathbb E}}
\nc{\bbF}{{\mathbb F}}
\nc{\bbG}{{\mathbb G}}
\nc{\bbH}{{\mathbb H}}
\nc{\bbI}{{\mathbb I}}
\nc{\bbJ}{{\mathbb J}}
\nc{\bbK}{{\mathbb K}}
\nc{\bbL}{{\mathbb L}}
\nc{\bbM}{{\mathbb M}}
\nc{\bbN}{{\mathbb N}}
\nc{\bbO}{{\mathbb O}}
\nc{\bbP}{{\mathbb P}}
\nc{\bbQ}{{\mathbb Q}}
\nc{\bbR}{{\mathbb R}}
\nc{\bbS}{{\mathbb S}}
\nc{\bbT}{{\mathbb T}}
\nc{\bbU}{{\mathbb U}}
\nc{\bbV}{{\mathbb V}}
\nc{\bbW}{{\mathbb W}}
\nc{\bbX}{{\mathbb X}}
\nc{\bbY}{{\mathbb Y}}
\nc{\bbZ}{{\mathbb Z}}
\nc{\calA}{{\mathcal A}}
\nc{\calB}{{\mathcal B}}
\nc{\calC}{{\mathcal C}}
\nc{\calD}{{\mathcal D}}
\nc{\calE}{{\mathcal E}}
\nc{\calF}{{\mathcal F}}
\nc{\calG}{{\mathcal G}}
\nc{\calH}{{\mathcal H}}
\nc{\calI}{{\mathcal I}}
\nc{\calJ}{{\mathcal J}}
\nc{\calK}{{\mathcal K}}
\nc{\calL}{{\mathcal L}}
\nc{\calM}{{\mathcal M}}
\nc{\calN}{{\mathcal N}}
\nc{\calO}{{\mathcal O}}
\nc{\calP}{{\mathcal P}}
\nc{\calQ}{{\mathcal Q}}
\nc{\calR}{{\mathcal R}}
\nc{\calS}{{\mathcal S}}
\nc{\calT}{{\mathcal T}}
\nc{\calU}{{\mathcal U}}
\nc{\calV}{{\mathcal V}}
\nc{\calW}{{\mathcal W}}
\nc{\calX}{{\mathcal X}}
\nc{\calY}{{\mathcal Y}}
\nc{\calZ}{{\mathcal Z}}

\nc{\scrA}{{\mathscr A}}
\nc{\scrB}{{\mathscr B}}
\nc{\scrR}{{\mathscr R}}

\nc{\bnu}{{\bar{ \nu}}}

\nc{\olO}{\bar{\calO}}

\nc{\al}{{\alpha}} 
\nc{\be}{{\beta}}
\nc{\ga}{{\gamma}} \nc{\Ga}{{\Gamma}}
 \nc{\hGa}{\hat{\Gamma}}
\nc{\ve}{{\varepsilon}} 
\nc{\la}{{\lambda}} \nc{\La}{{\Lambda}}
\nc{\om}{\omega} \nc{\Om}{\Omega} 
\nc{\sig}{{\sigma}} \nc{\Sig}{{\Sigma}}

\nc{\tnb}{\psi_{\rm tame}}

\nc{\op}{{\on{op}}}
\nc{\ad}{{\on{ad}}}
\nc{\alg}{{\on{alg}}}
\nc{\Ad}{{\on{Ad}}}
\nc{\Adm}{{\on{Adm}}} \nc{\aff}{{\on{af}}}
\nc{\Aut}{{\on{Aut}}}
\nc{\Bun}{{\on{Bun}}}
\nc{\cha}{{\on{char}}}
\nc{\der}{{\on{der}}}
\nc{\Der}{{\on{Der}}}
\nc{\diag}{{\on{diag}}}
\nc{\End}{{\on{End}}}
\nc{\Fl}{{\calF\!\ell}}
\nc{\Gal}{{\on{Gal}}}
\nc{\Gr}{{\on{Gr}}}
\nc{\rH}{{\on{H}}}
\nc{\Hom}{{\on{Hom}}}
\nc{\IC}{{\on{IC}}}
\nc{\id}{{\on{id}}}
\nc{\Id}{{\on{Id}}}
\nc{\ind}{{\on{ind}}}
\nc{\Ind}{{\on{Ind}}}
\nc{\Lie}{{\on{Lie}}}
\nc{\Pic}{{\on{Pic}}}
\nc{\pr}{{\on{pr}}}
\nc{\Res}{{\on{Res}}}
\nc{\res}{{\on{res}}} \nc{\Sat}{{\on{Sat}}}
\nc{\s}{{\on{sc}}}
\nc{\drv}{{\on{der}}}
\nc{\sgn}{{\on{sgn}}}
\nc{\Spec}{{\on{Spec}}}\nc{\Spf}{\on{Spf}} 
\nc{\Sph}{\on{Sph}}
\nc{\St}{{\on{St}}}
\nc{\tr}{{\on{tr}}}
\nc{\Tr}{{\on{Tr}}}
\nc{\Mod}{{\mathrm{-Mod}}}
\nc{\Hilb}{{\on{Hilb}}} 
\nc{\Ext}{{\on{Ext}}} 
\nc{\vs}{{\on{Vec}}}
\nc{\ev}{{\on{ev}}}
\nc{\nO}{{\breve{\calO}}}
\nc{\tS}{{\tilde{S}}}
\nc{\spe}{{\on{sp}}}

\nc{\nscrR}{{\mathscr{R}^{\on{nr}}}}

\nc{\GL}{{\on{GL}}}
\nc{\U}{{\on{U}}}
\nc{\Gl}{\on{Gl}} 
\nc{\GSp}{{\on{GSp}}}
\nc{\gl}{{\frakg\frakl}}
\nc{\SL}{{\on{SL}}} 
\nc{\SU}{{\on{SU}}} 
\nc{\SO}{{\on{SO}}}

\nc{\Conv}{{\on{Conv}}}
\nc{\Rep}{{\on{Rep}}}
\nc{\Dom}{{\on{Dom}}}
\nc{\red}{{\on{red}}}
\nc{\act}{{\on{act}}}
\nc{\nr}{{\on{nr}}}

\nc{\str}{{\on{-}}} 
\nc{\os}{{\bar{s}}}
\nc{\oeta}{{\bar{\eta}}}

\nc{\hookto}{\hookrightarrow}
\nc{\longto}{\longrightarrow}
\nc{\leftto}{\leftarrow}
\nc{\onto}{\twoheadrightarrow}
\nc{\lonto}{\twoheadleftarrow}

\nc{\bio}{{\bar{i}}}
\nc{\bjay}{{\bar{j}}}

\nc{\bFl}{{\overline{\Fl}}} 
\nc{\bU}{{\overline{U}}}
\nc{\tGr}{{\tilde{\Gr}}}
\nc{\cGr}{\calG\! r}
\nc{\oGr}{\overline{\on{Gr}}} 
\nc{\ocGr}{\overline{\calG\! r}}

\nc{\ohtimes}{\stackrel{!}{\otimes}}
\nc{\boxtilde}{\widetilde{\boxtimes}}
\nc{\vstar}{{\varhexstar}}

\nc{\bslash}{\backslash}
\nc{\algQl}{{\bar{\bbQ}_\ell}}
\nc{\sF}{{\bar{F}}}
\nc{\nF}{{\breve{F}}}
\nc{\nW}{{W^{\on{nr}}}}
\nc{\sk}{{\bar{k}}}
\nc{\cont}{\on{c}}
\nc{\supp}{\on{supp}}
\nc{\blt}{\bullet}  
\nc{\dom}{\on{dom}}
\nc{\scon}{{\on{sc}}} 
\nc{\Affine}{\on{Aff}} 
\nc{\nscrA}{\mathscr{A}^{\on{nr}}} 
\nc{\nfraka}{{\fraka^{\on{nr}}}}
\nc{\ran}{{\rangle}}
\nc{\lan}{{\langle}}
\nc{\bk}{{\bar{k}}}
\nc{\tF}{{\tilde{F}}}
\nc{\LG}{{^\text{L}\hspace{-0.04cm}G}}
\nc{\LL}{{^\text{L}\hspace{-0.07cm}L}}

\nc{\pot}[1]{ [\hspace{-0,5mm}[ {#1} ]\hspace{-0,5mm}] }
\nc{\rpot}[1]{ (\hspace{-0,7mm}( {#1} )\hspace{-0,7mm}) }

\nc{\defined}{\hspace{0.1cm}\stackrel{\text{\tiny def}}{=}\hspace{0.1cm}}

\begin{document}

\title[Iwahori Weyl]{On the Iwahori Weyl group}
\author[T. Richarz]{by Timo Richarz}

\address{Timo Richarz: Mathematisches Institut der Universit\"at Bonn, Endenicher Allee 60, 53115 Bonn, Germany}
\email{richarz@math.uni-bonn.de}

\maketitle
\setcounter{section}{1}

Let $F$ be a discretely valued complete field with valuation ring $\calO_F$ and perfect residue field $k$ of cohomological dimension $\leq 1$. In this paper, we generalize the Bruhat decomposition in Bruhat and Tits \cite{BT2} from the case of simply connected $F$-groups to the case of arbitrary connected reductive $F$-groups. If $k$ is algebraically closed, Haines and Rapoport \cite{HR} define the Iwahori-Weyl group, and use it to solve this problem. Here we define the Iwahori-Weyl group in general, and relate our definition of the Iwahori-Weyl group to that of \cite{HR}. Furthermore, we study the length function on the Iwahori-Weyl group, and use it to determine the number of points in a Bruhat cell, when $k$ is a finite field. Except for Lemma \ref{kernkott} below, the results are independent of \cite{HR}, and are directly based on the work of Bruhat and Tits \cite{BT1}, \cite{BT2}. 

\begin{ack}
I thank my advisor M. Rapoport for his steady encouragement and advice during the process of writing. I am grateful to the stimulating working atmosphere in Bonn and for the funding by the Max-Planck society.
\end{ack}

Let $\sF$ be the completion of a separable closure of $F$. Let $\nF$ be the completion of the maximal unramified subextension with valuation ring $\calO_\nF$ and residue field $\bar{k}$. Let $I=\Gal(\sF/\nF)$ be the inertia group of $\nF$, and let $\Sig=\Gal(\nF/F)$.

Let $G$ be a connected reductive group over $F$, and denote by $\scrB=\scrB(G,F)$ the enlarged Bruhat-Tits building. Fix a maximal $F$-split torus $A$. Let $\scrA=\mathscr{A}(G,A,F)$ be the apartment of $\scrB$ corresponding to $A$. 

\subsection{Definition of the Iwahori-Weyl group} Let $M=Z_G(A)$ be the centralizer of $A$, an anisotropic group, and let $N=N_G(A)$ be the normalizer of $A$. Denote by $W_0=N(F)/M(F)$ the relative Weyl group. 

\begin{dfn}\label{IwahoriWeyl}
i) The \emph{Iwahori-Weyl group} $W=W(G,A,F)$ is the group
\[
W\defined N(F)/M_1,
\]
where $M_1$ is the unique parahoric subgroup of $M(F)$. \smallskip \\
ii) Let $\fraka\subset\scrA$ be a facet and  $P_\fraka$ the associated parahoric subgroup. The subgroup $W_\fraka$ of the Iwahori-Weyl group corresponding to $\fraka$ is the group
\[
W_\fraka\defined P_\fraka\cap N(F)/M_1.
\]
\end{dfn}

The group $N(F)$ operates on $\scrA$ by affine transformations
\begin{equation}
\nu: N(F)\longrightarrow \Affine(\scrA). \label{apartact}
\end{equation} 
The kernel $\ker(\nu)$ is the unique maximal compact subgroup of $M(F)$ and contains the compact group $M_1$. Hence, the morphism \eqref{apartact} induces an action of $W$ on $\scrA$. 

Let $G_1$ be the subgroup of $G(F)$ generated by all parahoric subgroups, and define $N_1=G_1\cap N(F)$. Fix an alcove $\fraka_C\subset \scrA$, and denote by $B$ the associated Iwahori subgroup. Let $\mathbb{S}$ be the set of simple reflections at the walls of $\fraka_C$. By Bruhat and Tits \cite[Prop. 5.2.12]{BT2}, the quadruple
\begin{equation}
(G_1,B,N_1,\mathbb{S})  \label{doublesystem}
\end{equation}
is a (double) Tits system with affine Weyl group $W_{\aff}=N_1/N_1\cap B$, and the inclusion $G_1\subset G(K)$ is $B$-$N$-adapted of connected type. 

\begin{lem} \label{techlem}
i) There is an equality $N_1\cap B=M_1$. \smallskip \\
ii) The inclusion $N(F)\subset G(F)$ induces a group isomorphism $N(F)/N_1\xrightarrow{\cong} G(F)/G_1$.
\end{lem}

\begin{proof}
The group $N_1\cap B$ operates trivially on $\scrA$ and so is contained in $\ker(\nu)\subset M(F)$. In particular, $N_1\cap B=M(F)\cap B$. But $M(F)\cap B$ is a parahoric subgroup of $M(F)$ and therefore equal to $M_1$. \\
The group morphism $N(F)/N_1\rightarrow G(F)/G_1$ is injective by definition. We have to show that $G(F)=N(F)\cdot G_1$. This follows from the fact that the inclusion $G_1\subset G(F)$ is $B$-$N$-adapted, cf. \cite[4.1.2]{BT1}.
\end{proof}

Kottwitz defines in \cite[\S 7]{Kott} a surjective group morphism
\begin{equation}\label{Kott}
\kappa_G: G(F)\longrightarrow X^*(Z(\hat{G})^I)^\Sig.
\end{equation}
Note that in [\emph{loc. cit.}] the residue field $k$ is assumed to be finite, but the arguments extend to the general case.

\begin{lem} \label{kernkott} There is an equality $G_1=\ker(\kappa_G)$ as subgroups of $G(F)$.
\end{lem}
\begin{proof}
For any facet $\fraka$, let $\text{Fix}(\fraka)$ be the subgroup of $G(F)$ which fixes $\fraka$ pointwise. The intersection $\text{Fix}(\fraka)\cap \ker(\kappa_G)$ is the parahoric subgroup associated to $\fraka$, cf. \cite[Proposition 3]{HR}. This implies $G_1\subset \ker(\kappa_G)$. For any facet $\fraka$, let $\text{Stab}(\fraka)$ be the subgroup of $G(F)$ which stabilizes $\fraka$. Fix an alcove $\fraka_C$. There is an equality
\begin{equation}\label{fixstab}
\text{Fix}(\fraka_C)\cap G_1= \text{Stab}(\fraka_C) \cap G_1,
\end{equation}
and \eqref{fixstab} holds with $G_1$ replaced by $\ker(\kappa_G)$, cf. \cite[Lemma 17]{HR}. Assume that the inclusion $G_1\subset \ker(\kappa_G)$ is strict, and let $\tau\in \ker(\kappa_G)\backslash G_1$. By Lemma \ref{techlem} ii), there exists $g\in G_1$ such that $\tau'=\tau\cdot g$ stabilizes $\fraka_C$, and hence $\tau'$ is an element of the Iwahori subgroup $\text{Stab}(\fraka_C) \cap \ker(\kappa_G)$. This is a contradiction, and proves the lemma.
\end{proof}

By Lemma \ref{techlem}, there is an exact sequence
\begin{equation}\label{exactaf}
1\longrightarrow W_{\aff} \longrightarrow W\overset{\kappa_G}{\longrightarrow} X^*(Z(\hat{G})^I)^\Sig\longrightarrow 1.
\end{equation}
The stabilizer of the alcove $\fraka_C$ in $W$ maps isomorphically onto $X^*(Z(\hat{G})^I)^\Sig$ and presents $W$ as a semidirect product
\begin{equation}\label{semidirect}
W= X^*(Z(\hat{G})^I)^\Sig\ltimes W_{\aff}.
\end{equation}
For a facet $\fraka$ contained in the closure of $\fraka_C$, the group $W_\fraka$ is the parabolic subgroup of $W_{\aff}$ generated by the reflections at the walls of $\fraka_C$ which contain $\fraka$.

\begin{thm}\label{doubleclasses}
Let $\fraka$ (resp. $\fraka'$) be a facet contained in the closure of $\fraka_C$, and let $P_\fraka$ (resp. $P_{\fraka'}$) be the associated parahoric subgroup. There is a bijection
\begin{align*}
W_{\fraka}\backslash W/W_{\fraka'} & \overset{\cong}{\longrightarrow}P_{\fraka}\backslash G(F)/P_{\fraka'} \\
W_{\fraka}\,w\,W_{\fraka'} & \longmapsto P_{\fraka}\,n_w\,P_{\fraka'},
\end{align*}
where $n_w$ denotes a representative of $w$ in $N(F)$.
\end{thm}

\begin{proof}
Conjugating with elements of $N(F)$ which stabilize the alcove $\fraka_C$, we are reduced to proving that  
\begin{equation}\label{reduct}
W_{\fraka}\backslash W_{\aff}/W_{\fraka'} \longrightarrow P_{\fraka}\backslash G_1/P_{\fraka'}
\end{equation}
is a bijection. But \eqref{reduct} is a consequence of the fact that the quadruple \eqref{doublesystem} is a Tits system, cf. \cite[Chap. IV, \S 2, n$^o$ 5, Remark. 2)]{Bou}.
\end{proof}

\begin{rmk}\label{sccover}
Let $G_{\scon}\rightarrow G_{\drv}$ be the simply connected cover of the derived group $G_{\drv}$ of $G$, and denote by $A_{\scon}$ the preimage of the connected component $(A\cap G_{\drv})^0$ in $G_{\scon}$. Then $A_{\scon}$ is a maximal $F$-split torus of $G_{\scon}$. Let $W_{\scon}=W(G_{\scon},A_{\scon})$ be the associated Iwahori-Weyl group. Consider the group morphism $\varphi: G_{\scon}(F)\rightarrow G_{\drv}(F)\subset G(F)$. Then $G_1=\varphi(G_{\scon}(F))\cdot M_1$ by the discussion above \cite[Proposition 5.2.12]{BT2}, and this yields an injective morphism of groups
\[
W_{\scon} \longto W
\]
which identifies $W_{\scon}$ with $W_{\aff}$.
\end{rmk}

\subsection{Passage to $\nF$}   

Let $S$ be a maximal $\nF$-split torus which is defined over $F$ and contains $A$, cf. \cite{BT2}. Denote by $\nscrA=\scrA(G,S,\nF)$ the apartment corresponding to $S$ over $\nF$. The group $\Sig$ acts on $\nscrA$, and there is a natural $\Sig$-equivariant embedding 
\begin{equation}\label{fixapart}
\scrA\longrightarrow\nscrA,
\end{equation}
which identifies $\scrA$ with the $\Sig$-fixpoints $(\nscrA)^\Sig$, cf. \cite[5.1.20]{BT2}. The facets of $\scrA$ correspond to the $\Sig$-invariant facets of $\nscrA$.

Let $T=Z_G(S)$ (a maximal torus) be the centralizer of $S$, and let $N_S=N_G(S)$ be the normalizer of $S$. Let $T_1^{\text{nr}}$ be the unique parahoric subgroup of $T(\nF)$. Denote by $\nW=W(G,S,\nF)$ the Iwahori-Weyl group
\[
\nW=N_S(\nF)/T_1^{\nr}
\]
 over $\nF$. The group $\Sig$ acts on $\nW$, and the group of fixed points $(\nW)^\Sig$ acts on $\scrA$ by \eqref{fixapart}. We have
\[
(\nW)^\Sig=N_S(F)/T_1,
\]
since $H^1(\Sig,T_1^{\on{nr}})$ is trivial. For an element $n\in N_S(F)$ the tori $A$ and $nAn^{-1}$ are both maximal $F$-split tori of $S$ and hence are equal. This shows $N_S(F)\subset N(F)$, and we obtain a group morphism
\begin{equation}\label{fixiwahori}
(\nW)^\Sig=N_S(F)/T_1\longrightarrow N(F)/M_1=W,
\end{equation}
which is compatible with the actions on $\scrA$.

\begin{lem}\label{isolem}
The morphism \eqref{fixiwahori} is an isomorphism, i.e. $(\nW)^\Sig\overset{\cong}{\longrightarrow}W$.
\end{lem}

\begin{proof}
 Let $\fraka_C$ be a $\Sig$-invariant alcove of $\nscrA$. The morphism \eqref{fixiwahori} is compatible with the semidirect product decomposition \eqref{semidirect} given by $\fraka_C$. We are reduced to proving that the morphism
\[
(W_{\aff}^{\nr})^\Sig\longrightarrow W_{\aff}
\]
is an isomorphism. It is enough to show that $(W_{\aff}^{\nr})^\Sig$ acts simply transitively on the set of alcoves of $\scrA$. Let $\fraka_{C'}$ another $\Sig$-invariant alcove of $\nscrA$. Then there is a unique $w\in W_{\aff}^{\nr}$ such that $w\cdot\fraka_C=\fraka_{C'}$. The uniqueness implies $w\in(W_{\aff}^{\nr})^\Sig$.
\end{proof} 

\begin{cor}
Let $\fraka\subset\scrA$ be a facet, and denote by $\nfraka\subset\nscrA$ the unique facet containing $\fraka$. Then $W_\fraka=(W_{\nfraka})^\Sig$ under the inclusion $W\hookto \nW$.
\end{cor}
\vspace{-0.4cm}
\hfill\ensuremath{\Box}
 
\subsection{The length function on $W$} Let $\scrR=\scrR(G,A,F)$ be the set of affine roots. We regard $\scrR$ as a subset of the affine functions on $\scrA$. The Iwahori-Weyl group $W$ acts on $\scrR$ by the formula
 \begin{equation}
 (w\cdot \alpha)(x)=\alpha(w^{-1}\cdot x)
 \end{equation}
 for $w\in W$, $\alpha\in \scrR$ and $x\in\scrA$. This action preserves non-divisible\footnote{An element $\alpha\in\scrR$ is called non-divisible, if ${1\over 2}\alpha \not\in \scrR$.} roots. 
  
 Fix an alcove $\fraka_C$ in $\scrA$. By \eqref{semidirect}, $W$ is the semidirect product of $W_{\aff}$ with the stabilizer of the alcove $\fraka_C$ in $W$. Hence, $W$ is a quasi-Coxeter system and is thus equipped with a Bruhat-Chevalley partial order $\leq$ and a length function $l$. 
 
For $\alpha\in\scrR$, we write $\alpha>0$ (resp. $\alpha<0$), if $\alpha$ takes positive (resp. negative) values on $\fraka_C$. For $w\in W$, define
\begin{equation}
\scrR(w)\defined\{\alpha\in\scrR\;|\;\alpha>0 \quad\text{and}\quad w\alpha< 0\}.
\end{equation}  
We have $\scrR(w)=\scrR(\tau w)$ for any $\tau$ in the stabilizer of $\fraka_C$. Let $\mathbb{S}$ be the reflections at the walls of $\fraka_C$. These are exactly the elements in $W_{\aff}$ of length $1$. For any $s\in\mathbb{S}$, there exists a unique non-divisible root $\alpha_s\in\scrR(s)$. In particular, $\scrR(s)$ has cardinality $\leq 2$. 

\begin{lem}\label{positivlem}
Let $w\in W$ and $s\in\mathbb{S}$. If $\alpha\in\scrR(s)$, then $w\alpha>0$ if and only if $w\leq ws$.
\end{lem}

\begin{proof}
We may assume that $w\in W_{\aff}$ and that $\alpha_s=\alpha$ is non-divisible. We show that $w\alpha_s<0$ if and only if $ws\leq w$. If $w\alpha_s<0$, then fix a reduced decomposition $w=s_1\cdot\ldots\cdot s_n$ with $s_i\in\mathbb{S}$. There exists an index $i$ such that
\[
s_{i+1}\cdot\ldots\cdot s_n\alpha_s> 0 \qquad\text{and}\qquad s_i\cdot s_{i+1}\cdot\ldots\cdot s_n\alpha_s< 0,
\]
i.e. $s_{i+1}\cdot\ldots\cdot s_n\alpha_s=\alpha_{s_i}$ is the unique non-divisible root in $\scrR(s_i)$. Hence, 
\[
s_{i+1}\cdot\ldots\cdot s_n\cdot s\cdot s_n\cdot\ldots\cdot s_{i+1}=s_i,
\]
and $ws\leq w$ holds true. Conversely, if $ws\leq w$, then $ws\leq (ws)s$. This implies that $(ws)\alpha_s>0$ by what we have already shown. But $s\alpha_s=-\alpha_s$, and $w\alpha_s<0$ holds true.
\end{proof}

\begin{lem}\label{rootlem}
Let $w$, $v\in W$. Then
\[
\scrR(wv)\subset\scrR(v)\sqcup v^{-1}\scrR(w),
\]
and equality holds if and only if $l(wv)=l(w)+l(v)$.
\end{lem}
\begin{proof}
We may assume that $w$, $v\in W_{\aff}$. Assume that $s=v\in \mathbb{S}$, which will imply the general case by induction on $l(v)$. The inclusion 
\begin{equation}\label{rootlemproof}
\scrR(ws)\subset\scrR(s)\sqcup s\scrR(w) 
\end{equation}
is easy to see. If in \eqref{rootlemproof} equality holds, then we have to show that $l(ws)=l(w)+1$, i.e. $w\leq ws$. In view of Lemma \ref{positivlem}, it is enough to show that $w\alpha>0$ for $\alpha\in\scrR(s)$. But this is equivalent to $\scrR(s)\subset\scrR(ws)$, and we are done. Conversely, if $w\leq ws$, then equality in \eqref{rootlemproof} also follows from Lemma \ref{positivlem}.
\end{proof}

\begin{cor}
If every root in $\scrR$ is non-divisible, then $l(w)=|\scrR(w)|$ for every $w\in W$.
\end{cor}
\vspace{-0.2cm}
\hfill\ensuremath{\Box} 
 
\subsection{The length function on $W^{\text{nr}}$} In this section, the residue field $k$ is finite of cardinality $q$. Let $\nscrR=\scrR(G,S,\nF)$ be the set of affine roots over $\nF$. Note that every root of $\nscrR$ is non-divisible, since $G\otimes \nF$ is residually split. Let $\nW$ be the Iwahori-Weyl group over $F ^{\nr}$. Denote by $\fraka_C^{\nr}$ the unique $\Sig$-invariant facet of $\nscrA$ containing $\fraka_C$. Let $\leq ^{\nr}$ be the corresponding Bruhat order and $l^{\nr}$ the corresponding length function on $\nW$. By Lemma \ref{isolem}, we may regard $W$ as the subgroup of $\nW$ whose elements are fixed by $\Sig$. 

Let $w\in W$. If $\alpha\in\nscrR(w)$, then its restriction to $\scrA$ is non-constant, and hence $\alpha\in\scrR$ by \cite[1.10.1]{Ti}. We obtain a restriction map
\begin{equation}\label{restriction}
\begin{aligned}
\nscrR(w) & \longrightarrow \scrR(w) \\
\alpha & \longmapsto \alpha|_{\scrA}.
\end{aligned}
\end{equation}

 \begin{prop}\label{unramredprop}
The inclusion $W\subset \nW$ is compatible with the Bruhat orders in the sense that for $w,w'\in W$ we have $w\leq w'$ if and only if $w\leq^{\nr}w'$, and $l(w)=0$ if and only if $l^{\nr}(w)=0$. For $w\in W$, there is the equality
\[
|BwB/B|=q^{l^{\nr}(w)},
\]
where $B$ is the Iwahori subgroup in $G(F)$ attached to $\fraka_C$.
\end{prop} 

\begin{proof}
We need some preparation. Let $w\in W$, $s\in\mathbb{S}$ with $w\leq ws$.
\begin{sublem}\label{preplem} 
There is an equality
\[
\nscrR(ws)=\nscrR(s)\sqcup s\nscrR(w).
\]
In particular, $l^{\nr}(ws)=l^{\nr}(w)+l^{\nr}(s)$.
\end{sublem}
\begin{proof}
By Lemma \ref{rootlem} applied to $\nscrR$, the inclusion `$\subset$' holds for general $w,s\in \nW$.\smallskip \\
There is the inclusion $\nscrR(s)\subset\nscrR(ws)$: If $\alpha\in\nscrR(s)$, then $\alpha|_{\scrA}\in\scrR(s)$ by \eqref{restriction}. Since $w\leq ws$, we have $w\cdot\alpha|_{\scrA}>0$ by Lemma \ref{positivlem}. So $ws\cdot\alpha|_{\scrA}<0$ which shows that $ws\cdot\alpha<0$.\\
The inclusion $s\nscrR(w)\subset\nscrR(ws)$ follows similarly.
\end{proof}

Sublemma \ref{preplem} implies that the inclusion $W\subset \nW$ is compatible with the Bruhat orders. To show the rest of the proposition, we may assume that $w\in W_{\aff}$. Fix a reduced decomposition $w=s_1\cdot\ldots\cdot s_n$ with $s_i\in\mathbb{S}$. By standard facts on Tits systems, the multiplication map
\begin{equation}\label{multmap}
Bs_1B\times^B\ldots\times^BBs_nB/B\longrightarrow BwB/B
\end{equation}
is bijective. In view of Sublemma \ref{preplem}, we reduce to the case that $n=1$, i.e. $s=w\in\mathbb{S}$ is a simple reflection. Let $\calB$ be the Iwahori group scheme over $\mathcal{O}_F$ corresponding to the Iwahori subgroup $B$, and denote by $\calP$ the parahoric group scheme corresponding to the parahoric subgroup $B\cup BsB$. Let $\bar{\calP}_{\red}$ be the maximal reductive quotient of $\calP\otimes k$. This is a connected reductive group over $k$ of semisimple $k$-rank $1$. The image of the natural morphism
\[
\calB\otimes k\longrightarrow \calP\otimes k \longrightarrow\bar{\calP}_{\red},  
\]
is a Borel subgroup $\bar{\calB}$ of $\bar{\calP}_{\red}$. This induces a bijection
\begin{equation}\label{redmorph}
P/B\longrightarrow \bar{\calP}_{\red}(k)/\bar{\calB}(k).
\end{equation}
By Lang's Lemma, we have $\bar{\calP}_{\red}(k)/\bar{\calB}(k)=(\bar{\calP}_{\red}/\bar{\calB})(k)$. Let $\os$ be the image of $s$ under \eqref{redmorph}, and denote by $C_\os$ the $\bar{\calB}$-orbit of $\os$ in the flag variety $\bar{\calP}_{\red}/\bar{\calB}$.  It follows that the image of $BsB/B$ under \eqref{redmorph} identifies with the $k$-points $C_\os(k)$. Note that $\lan\os\ran$ is the relative Weyl group of $\bar{\calP}_{\red}$ with respect to the reduction to $k$ of the natural $\mathcal{O}_F$-structure on $A$. Then $C_\os\simeq \bar{U}$ where $\bar{U}$ denotes the unipotent radical of $\bar{\calB}$. But $\bar{U}$ is an affine space and hence $|C_\os(k)|=q^{\dim(\bar{U})}$. On the other hand, 
\[
l^{\nr}(s)=|\nscrR(s)|=\dim(\bar{U}),
\]
where the last equality holds because $\nscrR(s)$ may be identified with the positive roots of $\bar{\calP}_{\red}\otimes \bar{k}$ with respect to $\overline{\calB}\otimes \bar{k}$. 
\end{proof}

\begin{rmk}
i) If $G$ is residually split, then $l(w)=l^{\nr}(w)$ for all $w\in W$. \smallskip\\
ii) Tits attaches in \cite[1.8]{Ti} to every vertex $v$ of the local Dynkin diagram a positive integer $d(v)$. To the vertex $v$, there corresponds a non-divisible affine root $\alpha_v\in\scrR$, and a simple reflection $s_v\in\mathbb{S}$. Then Proposition \ref{unramredprop} shows that $d(v)=l^{\nr}(s_v)$, cf. \cite[3.3.1]{Ti}.  
\end{rmk}

\end{document}